\documentclass[12pt]{amsart}

\usepackage{amssymb,latexsym, amscd,pb-diagram}
\usepackage{cases}
\usepackage{graphicx}
\usepackage{amssymb}
\usepackage{amsfonts}
\usepackage{amsmath}
\usepackage{amscd}
\usepackage[square, numbers, sort&compress]{natbib}

\usepackage{hyperref}
\usepackage{xcolor}
\usepackage[all]{xy}
\usepackage{tikz-cd}
\usepackage{geometry}
\usepackage{arydshln}
\usepackage[T1]{fontenc}
\usepackage[utf8]{inputenc}

 \newtheorem{thm}{Theorem}[section]
 \newtheorem{problem}{Problem}
 \newtheorem{cor}[thm]{Corollary}
 \newtheorem{lem}[thm]{Lemma}
 \newtheorem{prop}[thm]{Proposition}
  \newtheorem{conj}[thm]{Conjecture}
  \theoremstyle{definition}
 \newtheorem{defn}[thm]{Definition}

 \theoremstyle{remark}
 \newtheorem{rem}{Remark}
 \newtheorem{example}{Example}[section]
 \numberwithin{equation}{section}
 
\begin{document}
\title[ Reduced characteristic number criteria for equivariant bordisms ]
 {Reduced characteristic number criteria for $T^k$- and \((\mathbb Z_2)^k\)-equivariant bordisms}
\author{ Runze Chen,  Zhi L\"u and Leqi Yang}

\address{School of Mathematical Science, Fudan University, Shanghai, China} \email{20110180032@fudan.edu.cn}
\address{School of Mathematical Science, Fudan University, Shanghai, China}
\email{zlu@fudan.edu.cn}
\address{School of Mathematical Science, Fudan University, Shanghai, China}
\email{24110180058@m.fudan.edu.cn}


\begin{abstract}
Classical  equivariant bordism theories require computing the full collection of equivariant characteristic numbers to detect whether an equivariant manifold bounds equivariantly or not.
This paper establishes simplified equivariant bordism characterizations
for two families of equivariant manifolds with isolated fixed points: unitary
\(T^k\)-manifolds and
closed smooth \((\mathbb{Z}_2)^k\)-manifolds.
For any  unitary $T^k$-manifold $M$ with isolated fixed points, we establish an equivariant unitary bordism criterion built entirely from
 a single polynomial of equivariant Chern classes.
We further introduce the minimal distinguishing degree and obtain two key inequalities that 
capture the interplay between $\dim M$ and the Euler characteristic $\chi(M)$ through this minimal distinguishing degree. These inequalities settle the existence problem
 of a linear lower bound for \(\chi(M)\) within the framework of Kosniowski’s conjecture and  partially verify the conjecture under natural admissible assumptions.
We also provide an alternative proof settling  the toric generalization of Kosniowski’s conjecture when \(\dim M=2k\).
By contrast, for a closed smooth $(\mathbb{Z}_2)^k$-manifold with isolated fixed points, we  derive a more concise equivariant bordism criterion relying solely on the powers of the top equivariant Stiefel–Whitney class.
Our new criteria substantially reduce computational demands.


 \end{abstract}

\maketitle

\section{Introduction}
Equivariant bordism, originating from foundational work of Conner and Floyd in the 1960s~\cite{CF1, CF, CF2}, is an important research area in geometric topology and transformation group theory, focusing on classifying smooth closed manifolds with compact Lie group actions. Equivariant characteristic numbers, including equivariant Chern numbers and equivariant Stiefel–Whitney numbers, serve as the fundamental  invariants for distinguishing equivariant bordism classes. Two major families of group actions have been widely studied: the torus group $T^k=(S^1)^k$ acting on orientable or unitary manifolds, and elementary abelian 2-group  $(\mathbb{Z}_2)^k$ acting  on unoriented closed smooth manifolds (see, e.g., \cite{CF,  St, tD1, Da,  K,  BV2, AB,  DJ, AP, F, GKM, S, G,  H, BP1, T}).
Localization formulas, which recover global equivariant characteristic numbers from local fixed-point data, have become   powerful computational tools throughout the history of this subject. Among them, there are two fundamental localization formulas: the Atiyah–Bott–Berline–Vergne localization formula for torus actions  and the tom Dieck–Kosniowski–Stong localization formula for $(\mathbb{Z}_2)^k$-actions. We emphasize that the study of equivariant manifolds with isolated fixed points bears distinctive intrinsic features and theoretical profundity, giving rise to numerous theoretical frameworks and research branches,  including toric varieties in algebraic geometry, toric geometry, DJ theory, GKM theory   and toric topology~\cite{ Da, DJ, F,  GKM,BP1}.

As a pioneer of modern equivariant bordism theory,  tom Dieck established a series of fundamental results for both real and complex equivariant bordism in the 1970s, laying the theoretical framework for later research (e.g., see~\cite{tD3, tD1, tD2, BD}).
In the setting of equivariant unoriented  bordism for $(\mathbb{Z}_2)^k$-manifolds, tom Dieck proved a landmark theorem: the equivariant unoriented bordism class of any smooth closed $(\mathbb{Z}_2)^k$-manifold is completely determined by all its  equivariant Stiefel–Whitney numbers.
Meanwhile, he extended the framework to unitary $G$-manifolds
 with $G=T^k\times\mathbb{Z}_m$,
 proving that equivariant K-theoretic Chern numbers can completely characterize their equivariant unitary bordism classes.

Subsequent breakthroughs focused specifically on unitary \(T^k\)-manifolds and their equivariant cohomology Chern numbers. Guillemin, Ginzburg and Karshon  proved in \cite{G} that for a unitary \(T^k\)-manifold with isolated fixed points (i.e., the fixed point set is  nonempty and finite), its equivariant unitary bordism class is completely characterized by all equivariant cohomology Chern numbers. They further conjectured that this  characterization still holds even without the isolated fixed point hypothesis, and Lü and Wang provided an affirmative solution to this conjecture in \cite{LW}.

These foundational results  provide the general framework for further research, yet they require computing the complete collection  of  equivariant characteristic numbers, imposing heavy computational demands.

Several partial simplifications have been achieved for unitary $T^k$-manifolds with isolated fixed points in specific dimensional cases. Note (cf \cite{K}) that any such $T^k$-manifold with a finite nonempty  fixed point set is even-dimensional and satisfies $\dim M\ge 2k$.
When $\dim M=2k$, Lü and Tan showed in \cite{LT1} that only the equivariant cohomology Chern numbers constructed from the first and second  equivariant Chern classes are sufficient to determine the equivariant unitary bordism class of $M$. For $\dim M=2(k-1)$,  Wen and Ma proved in \cite{WM}  that the first six equivariant Chern classes can play the same role as the case $\dim M=2k$.
These case-specific results strongly suggest that, for unitary \(T^k\)-manifolds with isolated fixed points, we do not need the entire family of equivariant cohomology Chern numbers to judge whether a unitary $T^k$-manifold is an equivariant boundary. This naturally raises the following problem:

\begin{problem}\label{problem}
For arbitrary $2n$-dimensional unitary $T^k$-manifold with isolated fixed points, what minimal collection of equivariant cohomology  Chern numbers can fully characterize its equivariant unitary bordism class?
\end{problem}

We  establish an equivariant unitary bordism criterion based upon a single distinguishing polynomial, which is an equivariant cohomology polynomial in equivariant Chern classes satisfying local fixed-point conditions (see Definition~\ref{dis-poly-notion}).
Lemma~\ref{polynomial} guarantees the existence of such polynomials, which yields this reduced bordism characterization, thereby resolving Problem~\ref{problem}. We state our main result as follows.

\begin{thm}\label{main thm1}
Let $M^{2n}$ be a $2n$-dimensional unitary $T^k$-manifold with a finite fixed point set $M^{T^k}$. Then $M^{2n}$ bounds equivariantly if and only if there exists a distinguishing  polynomial $g$ such that $$\langle g^l, [M]\rangle = 0$$ for all integers $l < |M^{T^k}|$, where $[M]$ denotes the fundamental homology class of $M$.
\end{thm}

\begin{rem}
In Subsection~\ref{Example},
we construct an explicit example to show that the top equivariant Chern class alone cannot serve as a valid choice of $g$ (stated in Theorem~\ref{main thm1}) for $n>1$, which clarifies the essential limitation of using a single top equivariant Chern class.
In  Subsection~\ref {Case-k=n},  we  derive two admissible explicit expressions of $g$
 under the assumption \(k=n\),
which implies the non-uniqueness of $g$.
 \end{rem}

For  a $2n$-dimensional unitary $T^k$-manifold $M^{2n}$ with a finite fixed point set $M^{T^k}$,  we define the {\em  minimal distinguishing degree}
$${\bf m}(k, M^{2n})=\min\{\deg g\big| g \text{ is a distinguishing  polynomial}\},$$
which is an even positive integer 
by Lemma~\ref{polynomial}, which guarantees  the existence of  $g$, as mentioned before,   This minimal distinguishing degree depends on the manifold $M^{2n}$ as well as its $T^k$-action (see Subsection~\ref{Case-k=n}).

Recall from \cite{AP} that the Euler characteristic $\chi(M^{2n})=\chi(M^{T^k})$, so $\chi(M^{2n})=|M^{T^k}|$, which implies that the number $|M^{T^k}|$ is uniquely determined by $M^{2n}$ and independent of the $T^k$-action  since the Euler characteristic is a topological invariant.
Moreover, as a  direct consequence of Theorem~\ref{main thm1}, we obtain the following dimensionally constrained vanishing result using our minimal distinguishing degree.
\begin{cor}\label{md-deg inequ}
 Suppose that $M^{2n}$ is a  $2n$-dimensional unitary $T^k$-manifold with a finite fixed point set.
 If ${\bf m}(k, M^{2n})\chi(M^{2n})\leq 2n$, then $M^{2n}$ bounds equivariantly. 
\end{cor}

\begin{rem}
  Since $2n$ and $\chi(M^{2n})$ are intrinsic to $M^{2n}$, this means that $\boldsymbol{m}(k,M^{2n})$ encodes the core information of the $T^k$-action. In addition,  Corollary~\ref{md-deg inequ} also implies that  if $M^{2n}$ does not bound equivariantly, then     \begin{equation}\label{important-inequ} {\bf m}(k, M^{2n})\chi(M^{2n})> 2n.
    \end{equation}
  This captures the fundamental interplay between
$\chi(M)$
 and the dimension of
$M$,  mediated by
${\bf m}(k, M^{2n})$.  Applying
the inequality~\eqref{important-inequ}   to Kosniowski's conjecture and its toric generalization,
we  settle the existence problem
of a positive linear function in $n$ that furnishes a lower bound for $\chi(M^{2n})$  within the framework of Kosniowski’s conjecture (see Proposition~\ref{ex-lin-bound}).
\end{rem}

Under a suitable condition,  we obtain  a refined lower bound sharper than the inequality (\ref{important-inequ}), which is stated as follows.

\begin{cor}\label{sharp bound}
 Suppose   $M^{2n}$ is a  $2n$-dimensional unitary $T^k$-manifold with a finite fixed point set, satisfying further  that  $M^{2n}$ bounds non-equivariantly, but does not bound equivariantly. Then  \begin{equation}\label{sharp ineq}
   {\bf m}(k, M^{2n})\chi(M^{2n})>2n {}+{\bf m}(k, M^{2n}).\end{equation}
\end{cor}

\begin{rem}
In Section~\ref{apply-KC},  we apply Corollary \ref{sharp bound} to derive a  lower bound for $\chi(M^{2n})$ that improves the linear
 function $f(n)={n\over 2}$, the prime candidate predicted by Kosniowski, thus partially confirming Kosniowski's
conjecture and its toric generalization provided that a unitary \(T^k\)-manifold \(M^{2n}\) with a finite  fixed point set  bounds non-equivariantly but not equivariantly and satisfies  ${\bf m}(k, M^{2n})\leq 4$ (see Proposition~\ref{sharper bound}).
The proof of Corollary \ref{sharp bound} will be finished in Section~\ref{apply-KC}.
\end{rem}

In contrast to the complex setting with $T^k$-actions, the real counterpart equipped with \((\mathbb{Z}_2)^k\)-actions possesses much stronger structural properties: the top equivariant Stiefel–Whitney class alone is sufficient.

\begin{thm}\label{main thm2}
Let $M^n$ be an $n$-dimensional closed smooth $(\mathbb {Z}_2)^k$-manifold with a finite fixed point set $M^{(\mathbb {Z}_2)^k}$. Then $M^n$ bounds equivariantly if and only if
\[
\bigl\langle \bigl(w_n^{(\mathbb {Z}_2)^k}(TM)\bigr)^l,\,[M]\bigr\rangle = 0
\]
for all integers $l \le \bigl|M^{(\mathbb {Z}_2)^k}\bigr|$, where $w_n^{(\mathbb {Z}_2)^k}(TM)$ denotes the top equivariant Stiefel--Whitney class of the tangent bundle $TM$, and $[M]$ is the mod-$2$ fundamental homology class of $M$.
\end{thm}

\begin{rem}
Compared with the original requirement of computing all equivariant Stiefel--Whitney numbers, our criterion
greatly simplifies how to detect whether a  smooth closed $(\mathbb{Z}_2)^k$-manifold  with isolated fixed points is an equivariant  boundary. We also reveal that the paired distribution of equivariant Euler classes at fixed points is the essential geometric reason for this simplification, adding new interpretations to a classical theorem  of Stong~\cite{St}  in the case of isolated fixed points (see  Corollary~\ref{pair} and Remark~\ref{Stong}).

We see from  Theorems~\ref{main thm1} and~\ref{main thm2} that substantial distinctions emerge between the real and complex settings. This discrepancy originates chiefly from orientation variations among the normal  $T^k$-representations and the tangent $T^k$-representations at fixed points,  which explains why the top equivariant characteristic class plays a decisive role in the real setting but not generally in the complex setting. See Subsection~\ref{loc-for} for more details.

The  techniques employed throughout our proofs are mainly the  localization formulas and Vandermonde matrix arguments.

Finally, it is worth noting that our main results cannot be readily generalized to manifolds with infinite fixed point sets. The main difficulty is constructing distinguishing polynomials that discriminate all fixed-point components and their normal bundles, which requires both ordinary Chern classes of each component’s tangent bundle and equivariant Chern classes of its normal bundle. Nevertheless, for manifolds with  isolated fixed points, only the equivariant Chern classes of the tangent representation at each isolated fixed point are essentially relevant.
\end{rem}

This paper is organized as follows.  Section~\ref{review} reviews necessary preliminaries on  equivariant characteristic classes and their numbers, and the two localization formulas that form the foundation of our proofs.
Section~\ref{proof}
focuses on torus actions on unitary manifolds: we introduce an equivalence relation on fixed points, prove the existence of  distinguishing polynomials, establish our main bordism criterion (Theorem~\ref{main thm1}),  present an explicit counterexample demonstrating that the top equivariant Chern class alone fails to yield a complete invariant for \(n>1\), and analyze the special cases \(k=n-1, n\).
Section~\ref{proof2}
concerns closed smooth \((\mathbb{Z}_2)^k\)-manifolds. We prove the corresponding bordism criterion (Theorem~\ref{main thm2}) and explain the geometric origin of  the stronger simplification in the real setting.
Section~\ref{apply-KC} gives applications of our two key inequalities (\ref{important-inequ}) and (\ref{sharp ineq}) to  Kosniowski's conjecture and its toric generalization.

\section{Preliminaries}\label{review}

This paper mainly focuses on equivariant cohomology characteristic numbers and their associated techniques. Therefore, in this section, we only briefly recall the relevant definitions and known results (cf~\cite{tD1,    BV1, BV2, AB,  KS, AP,LW}).

\subsection{Equivariant cohomology  characteristic numbers}
\label{symm-Chern class}
A {\em unitary} manifold is defined as an oriented, connected, closed smooth manifold whose stable tangent bundle admits a complex structure. A unitary $T^k$-manifold is further required to allow  an effective action of $T^k$ that preserves this complex structure.
Let $M$ be a unitary $T^k$-manifold of dimension $2n$. Applying the Borel construction to
the  tangent bundle $TM$ over $M$ gives a vector bundle
$ET^k\times_{T^k} TM$ over
$ET^k\times_{T^k} M$, where $ET^k=(S^\infty)^k$ is the universal free $T^k$-space with orbit space
$BT^k=(\mathbb{C}P^\infty)^k$, the classifying space of $T^k$.
 Then the {\em total equivariant cohomology Chern class} of $TM$ is defined to be
the total cohomology Chern class of the vector bundle $ET^k\times_{T^k} TM$, i.e.,
\begin{align*}
c^{T^k}(TM):= c(ET^k\times_{T^k} TM)\in H^*(ET^k\times_{T^k} M)=H_{T^k}^*(M)
\end{align*}
which can formally be written as
$c^{T^k}(TM)=\prod_{i=1}^n(1+x_i)$ with $\deg x_j=2$ according to the splitting principle, so
the $i$-th equivariant cohomology Chern
class $c_i^{T^k}(TM)$ is  the $i$-th elementary symmetric function $\sigma_i(x_1, ..., x_n)$. For the definition of $\sigma_i$, see also ~\cite{MS1}.
\vskip .1cm

Now let $p_{!}^{T^k}: H_{T^k}^*(M)\longrightarrow H^{*-2n}(BT^k)$ be the equivariant Gysin map induced by the constant $p: M\longrightarrow {\rm pt}$. Then the {\em equivariant cohomology Chern numbers} of $M$ are defined to be
$$\langle c^{T^k}_\omega(TM), [M] \rangle: =p_{!}^{T^k}(c^{T^k}_\omega(TM)),
$$
which are polynomials in
$H^*(BT^k)=\mathbb{Z}[t_1, ..., t_k]$ with $\deg t_j=2$, where $\omega=(i_1, ..., i_u)$ is a partition of $|\omega|=i_1+\cdots+i_u$, and $c^{T^k}_\omega(TM)$ is the product $c_{i_1}^{T^k}(TM)\cdots c_{i_u}^{T^k}(TM)$.
More generally, for an arbitrarily symmetric polynomial function $f(x_1, ..., x_n)$,
$$\langle f(x_1, ..., x_n), [M] \rangle: =p_{!}^{T^k}(f(x_1, ..., x_n)).
$$

\vskip .1cm
It has been proved in~\cite{LW} that the $T^k$-equivariant cohomology Chern numbers form a full system of invariants of the $T^k$-equivariant unitary bordism.
\begin{thm}[{L\"u--Wang \cite{LW}}] \label{invariant1}
 Let $M$ be a unitary $T^k$-manifold. Then $M$ bounds equivariantly if and only if all   equivariant cohomology Chern numbers of $M$ vanish.
\end{thm}

Similarly, for a smooth closed $(\mathbb{Z}_2)^k$-manifold $M$ of dimension $n$, one may define the {\em total equivariant Stiefel--Whitney
class} of $M$ as
\begin{align*}
 w^{(\mathbb{Z}_2)^k}(TM): =
w(E(\mathbb{Z}_2)^k\times_{(\mathbb{Z}_2)^k}TM) \in
H^*(E(\mathbb{Z}_2)^k\times_{(\mathbb{Z}_2)^k}M; \mathbb{Z}_2)=
H^*_{(\mathbb{Z}_2)^k}(M; \mathbb{Z}_2)
\end{align*}
which can formally be written as
$w^{(\mathbb{Z}_2)^k}(TM)=\prod_{i=1}^n(1+x_i)$ with $\deg x_j=1$, so that
the $i$-th equivariant Stiefel--Whitney
class $w_i^{(\mathbb{Z}_2)^k}(TM)=\sigma_i(x_1, ..., x_n)$,
where $E(\mathbb{Z}_2)^k=(S^\infty)^k$ is the universal free $(\mathbb{Z}_2)^k$-space with orbit space
$B(\mathbb{Z}_2)^k=(\mathbb{R}P^\infty)^k$, the classifying space of $(\mathbb{Z}_2)^k$.
Then the {\em equivariant Stiefel--Whitney numbers} of $M$ are defined to be
$$\langle w_\omega^{(\mathbb{Z}_2)^k}(TM), [M]\rangle:=p_{!}^{(\mathbb{Z}_2)^k}(w_\omega^{(\mathbb{Z}_2)^k}(TM))
$$
which are polynomials in
$H^*(B(\mathbb{Z}_2)^k; \mathbb{Z}_2)=\mathbb{Z}_2[a_1, ..., a_k]$ with $\deg a_j=1$,
 where $p_{!}^{(\mathbb{Z}_2)^k}:
 H^*_{(\mathbb{Z}_2)^k}(M; \mathbb{Z}_2)\longrightarrow
 H^{*-n}(B(\mathbb{Z}_2)^k; \mathbb{Z}_2)$ is the equivariant Gysin map induced by mapping $M$ to a point,
  $\omega=(i_1, ..., i_u)$  is a partition, and
$w_\omega^{(\mathbb{Z}_2)^k}(TM)=w_{i_1}^{(\mathbb{Z}_2)^k}(TM)\cdots  w_n^{(\mathbb{Z}_2)^k}(TM)$.

\vskip .1cm
tom Dieck showed in~\cite{tD1} that the equivariant Stiefel--Whitney  numbers of $M$ completely determine the equivariant unoriented bordism class of $M$.
\begin{thm}[tom Dieck~\cite{tD1}]\label{invariant2}
Let $M$ be  a smooth closed $(\mathbb{Z}_2)^k$-manifold. Then   $M$ bounds equivariantly if and only if all its equivariant Stiefel--Whitney numbers vanish.
\end{thm}

\subsection{Localization formulae}
\label{loc-for}
Localization formulae furnish an effective approach for computing equivariant characteristic numbers from fixed-point data. To establish our main theorems, we require the localization formulae for both $T^k$-actions and $(\mathbb{Z}_2)^k$-actions with only isolated fixed points.

\subsubsection{Atiyah--Bott--Berline--Vergne localization formula}
As a core computational tool for $T^k$-actions,  the Atiyah–Bott–Berline–Vergne localization formula was  established independently by Berline and Vergne \cite{BV1,BV2} in the framework of equivariant differential forms, and by Atiyah and Bott \cite{AB} within Borel equivariant cohomology. A complete unified treatment of the theorem can be found in the monograph \cite{BGV}.

Suppose that $M^{2n}$ is a $2n$-dimensional unitary $T^k$-manifold with isolated fixed points (i.e., its fixed point set $M^{T^k}$ is finite and nonempty).
Take a fixed point $p\in M^{T^k}$, its tangent space $T_pM$ is naturally  regarded as a faithful $T^k$-representation since the $T^k$-action on $M^{2n}$ has been assumed to be effective.
On the other hand, the normal bundle $\nu_p$ to $p$ in $M^{2n}$ is $T_pM$ with the orientation inherited from $M^{2n}$.
Thus, the equivariant Euler class $e^{T^n}(\nu_p)$ of this bundle is $\varepsilon_p e^{T^n}(T_pM)$, where $\varepsilon_p$ is equal to $+1$ if the orientation of $\nu_p$ agrees with the complex orientation on the stable tangent bundle $TM$ and $-1$ otherwise.
Since all nontrivial irreducible
$T^k$-representations are complex 1-dimensional, the total equivariant Chern class
$c^{T^k}(T_pM)$ can be written as
$$c^{T^k}(T_pM)=\prod_{i=1}^n(1+\lambda_{p, i})$$
where each $\lambda_{p,i}$ is an element of degree 2 in $H^*(BT^k)=\mathbb{Z}[t_1, ..., t_k]$. Thus, the equivariant Euler class $$e^{T^n}(T_pM)=c_n^{T^n}(T_pM)=\lambda_{p, 1}\cdots\lambda_{p,n}=\sigma_n(\lambda_{p, 1}, ...,\lambda_{p,n}).$$

\vskip 0.1
cm

As seen in Subsection~\ref{symm-Chern class}, when  we  formally write  the total class
$c^{T^n}(TM)$ as $\prod_{i=1}^n(1+x_i)$, any symmetric polynomial
$f(x_1, ..., x_n)$ becomes  a cohomology class produced by equivariant cohomology Chern classes $\sigma_j(x_1, ..., x_n)$ in $H^*_{T^k}(M)$. It is well-known (cf~\cite{AP}) that this symmetric polynomial
$f(x_1, ..., x_n)$ restricted to $p$ is exactly $f(\lambda_{p, 1}, ..., \lambda_{p,n})$, i.e., the homomorphism $i_p^*: H^*_{T^k}(M)\longrightarrow H^*_{T^k}(\{p\})=H^*(BT^k)$ induced by the inclusion $i_p: \{p\}\hookrightarrow M$ maps $f(x_1, ..., x_n)$ to $f(\lambda_{p, 1}, ..., \lambda_{p,n})$.
Now let us state the
Atiyah--Bott--Berline--Vergne  (ABBV)  localization formula  as
follows:

\begin{thm}[ABBV localization formula]\label{ABBV}
Let~$M^{2n}$ be a $2n$-dimensional unitary  $T^k$-manifold with a finite fixed point set $M^{T^k}$.  Then
$$\big\langle f(x_1, ..., x_n), [M]\big\rangle=\sum_{p\in M^{T^k}}{\varepsilon_pi_p^*(f(x_1, ..., x_n)\over e^{T^k}(T_pM)}
=\sum_{p\in M^{T^k}}{\varepsilon_pf(\lambda_{p, 1}, ..., \lambda_{p,n})\over \sigma_n(\lambda_{p, 1}, ...,\lambda_{p,n})}
\in H^*(BT^k).$$
\end{thm}

\subsubsection{tom Dieck--Kosniowski--Stong localization formula}
In~\cite{KS},   Kosniowski and Stong gave the calculation formula of the equivariant Stiefel--Whitney numbers in terms of fixed-point  data for smooth closed $(\mathbb{Z}_2)^k$-manifolds,
which can also be induced from tom Dieck's  integrality theorem in \cite{tD1}, as shown in~\cite[Section 8]{LT2} for the case of $(\mathbb{Z}_2)^k$-actions fixing isolated points.

Let $M^n$ be an $n$-dimensional smooth closed $(\mathbb{Z}_2)^k$-manifold with a finite fixed point set $M^{(\mathbb{Z}_2)^k}$.
Similarly to the case of unitary $T^k$-manifolds, if
the total class $w^{(\mathbb{Z}_2)^k}(TM)$ is formally written as $\prod_{i=1}^n(1+x_i)$, then  an arbitrary  symmetric polynomial $f(x_1, ..., x_n)$ becomes a cohomology class produced by equivariant Stiefel--Whitney classes $\sigma_j(x_1, ..., x_n)$
in $H^*_{(\mathbb{Z}_2)^k}(M; \mathbb{Z}_2)$. Since all irreducible $(\mathbb{Z}_2)^k$-representations are 1-dimensional, take a fixed point $p\in M^{(\mathbb{Z}_2)^k}$,
we have that the total equivariant Stiefel--Whitney class
of the tangent representation $T_pM$ at $p$ can be written as
$\prod_{i=1}^n(1+\lambda_{p,i})$ where
each $\lambda_{p,i}$ is a nonzero element of degree one in $H^*(B(\mathbb{Z}_2)^k; \mathbb{Z}_2)$. Thus,
the equivariant Euler class of $T_pM$ is $$e^{(\mathbb{Z}_2)^k}(T_pM)=\prod_{i=1}^n\lambda_{p,i}=\sigma_n(\lambda_{p,1},..., \lambda_{p,n}).$$
We note that the equivariant Euler class $e^{(\mathbb{Z}_2)^k}(\nu_p)$ of the normal representation $\nu_p$ at $p$ is equal to $e^{(\mathbb{Z}_2)^k}(T_pM)$ because of the ignorance of orientation.

\begin{thm}[tom Dieck--Kosniowski--Stong localization formula]\label{DKS}
Suppose that $M^n$ is an $n$-dimensional smooth closed $(\mathbb{Z}_2)^k$-manifold with a finite  fixed point set $M^{(\mathbb{Z}_2)^k}$.  Then for any   symmetric polynomial $f(x_1, ..., x_n)$,
\begin{equation} \label{formula2}\big\langle f(x_1, ..., x_n), [M] \big\rangle
=\sum_{p\in M^{(\mathbb{Z}_2)^k}}
{f(\lambda_{p,1}, ..., \lambda_{p,n}) \over\sigma_n(\lambda_{p,1},..., \lambda_{p,n})}\in H^*(B(\mathbb{Z}_2)^k; \mathbb{Z}_2).\end{equation}
\end{thm}

\begin{rem}\label{deg<2n}
It should be pointed out that in Theorems~\ref{ABBV} and~\ref{DKS}, if $\deg f< \dim M$, then  $\big\langle f(x_1, ..., x_n), [M] \big\rangle
=0$.
\end{rem}

\subsection{The case of rank $k=1$}
\begin{cor} \label{even}
Let $n$ be an integer bigger than zero. If    an
  $n$-dimensional smooth closed $\mathbb{Z}_2$-manifold $M$ has only isolated fixed points, then the number of isolated fixed points must be even, and $M$ bounds equivariantly.
\end{cor}
\begin{proof}
Since \(H^\ast(B\mathbb {Z}_2;\mathbb {Z}_2)=\mathbb {Z}_2[a]\) is a polynomial algebra generated by a single degree-1 element $a$, the equivariant Euler class \(\sigma_n(\lambda_{p,1},\dots,\lambda_{p,n})\) of \(T_pM\) at each fixed point \(p\in M^{\mathbb {Z}_2}\) coincides with \(a^n\). Consequently, \(f(\lambda_{p,1},\dots,\lambda_{p,n})=f(\lambda_{q,1},\dots,\lambda_{q,n})\) for any two fixed points \(p,q\in M^{\mathbb {Z}_2}\)  in the formula \eqref{formula2}. Setting \(f(x_1,\dots,x_n)=1\) in \eqref{formula2} yields that \(|M^{\mathbb {Z}_2}|\) is even. Moreover,
\(\big\langle f(x_1,\dots,x_n),[M]\big\rangle=0\label{formula-2}\)
holds for any symmetric polynomial \(f(x_1,\dots,x_n)\), which implies that all equivariant Stiefel–Whitney numbers of $M$ vanish. By Theorem \ref{invariant2}, so $M$ bounds equivariantly.
\end{proof}

\begin{rem} \label{exk=1}
Corollary \ref{even} is a special case of the classical results due to Conner and Floyd (see \cite[Theorems 25.1, 25.2]{CF}), for which we provide an alternative proof using the localization formula here.

However, in contrast to Corollary \ref{even}, a unitary \(S^1\)-manifold $M$ 
can exactly have an odd number of isolated fixed points
and hence need not bound equivariantly  in general.
By \cite[Theorem IV.10.5]{B} and \cite[Lemma 7.4.3]{BP1}, abundant nonbounding examples with odd isolated fixed points can be constructed from (omnioriented) quasitoric manifolds, the topological versions of toric varieties,   introduced by Davis and Januszkiewicz \cite{DJ}. Notably, $M$ may still fail to bound equivariantly even when $|M^{S^1}|$ is even.
A canonical example is $\mathbb{C}P^1$ with the standard $S^1$-action, which has precisely two fixed points and is nonbounding.
Consequently, the classification of nonbounding unitary \(S^1\)-manifolds with finite fixed points is considerably intricate. This is also seen from Kosniowski’s conjecture.
\end{rem}

\section{Equivariant bordism for unitary \(T^k\)-manifolds
}\label{proof}

Let~$M^{2n}$ be a $2n$-dimensional unitary  $T^k$-manifold with a finite fixed point set $M^{T^k}$.
Take a point \(p\in M^{T^k}\). By the preliminaries in Section \ref{review}, the equivariant Euler class of the normal \(T^k\)-representation \(\nu_p\) and the top equivariant Chern class of  tangent representation \(T_pM\) at $p$ satisfy
$$e^{T^k}(\nu_p)=\varepsilon_p c_n^{T^k}(T_pM)=\varepsilon_p\prod_{i=1}^n\lambda_{p,i}\in H^\ast(BT^k)=\mathbb Z[t_1,\dots,t_k],$$
in which \(\varepsilon_p=\pm1\) is determined by the relative orientations of \(\nu_p\) and \(T_pM\).

\begin{defn}
 Define an equivalence relation $\sim$ on $M^{T^k}$ in such a way:
$$p\sim q \Longleftrightarrow
c^{T^k}(T_pM)= c^{T^k}(T_qM).$$
\end{defn}
Take an equivalence class $[p]$ in the quotient set $M^{T^k}/\sim$, we have the sum $\sum_{q\in[p]}\varepsilon_q$ of signs $\varepsilon_q, q\in[p]$, denoted by $\varepsilon_{[p]}$.
We note that if $c^{T^k}(T_pM)= c^{T^k}(T_qM)$, then  $\{\lambda_{p,1}, ..., \lambda_{p,n}\}=\{\lambda_{q,1}, ..., \lambda_{q,n}\}$. 
Furthermore,  the ABBV localization formula in Theorem~\ref{ABBV} becomes
\begin{align} \label{ABBV-1}
   \big\langle f(x_1, ..., x_n), [M]\big\rangle
=\sum_{[p]\in M^{T^k}/\sim}
{\varepsilon_{[p]}f(\lambda_{p, 1}, ..., \lambda_{p,n})\over \sigma_n(\lambda_{p, 1}, ...,\lambda_{p,n})}\end{align}
from which we easily see the following lemma.
\begin{lem}\label{zero}
 If $\varepsilon_{[p]}=0$  for all $[p]\in M^{T^k}/\sim$, then all equivariant cohomology Chern numbers of $M$ vanish.
\end{lem}

\begin{rem}
Lemma~3.2 also admits a geometric interpretation in terms of equivariant
bordism. Suppose that $\varepsilon_{[p]}=0$ for every equivalence class
$[p]\in M^{T^k}/\!\sim$. Then, within each equivalence class, the fixed
points with sign $+1$ can be paired with those with sign $-1$. For each
such pair $p$ and $q$, the tangent $T^k$-representations $T_pM$ and
$T_qM$ are isomorphic, while the orientations induced from $M$ differ from their complex orientations by opposite signs.

Choose sufficiently small disjoint invariant disk neighborhoods of $p$ and $q$, identified equivariantly with disks in their common tangent $T^k$-representation. In the $(2n+1)$-dimensional bordism $M\times[0,1]$, one may join these two neighborhoods by an equivariant tube $D^{2n} \times [0,1]$. The boundary of this tube consists of the two disk neighborhoods at its ends together with its cylindrical boundary. Thus, on the outgoing boundary of the resulting bordism, the two invariant disk neighborhoods containing $p$ and $q$ are replaced by the cylindrical part of the tube. Since the $T^k$-action on this cylindrical part has no fixed points, the pair $p,q$ is thereby eliminated. This is precisely the local equivariant connected sum construction at the two fixed points.

Performing this construction for all pairs yields a unitary
$T^k$-manifold equivariantly bordant to $M$ and having no $T^k$-fixed
points. Equivalently, the resulting bordism class has trivial fixed-point
data. Since the fixed-point homomorphism in $T^k$-equivariant
unitary bordism theory is injective (see \cite{H}), this class
is zero. Thus $M$ bounds equivariantly. This gives a geometric
explanation, independent of the localization formula, for the conclusion
underlying Lemma~3.2.
\end{rem}

Formula \eqref{ABBV-1} constitutes the central tool in our proof of Theorem 1.1. To this end, we require  symmetric polynomials
capable of distinguishing different equivalence classes  in $M^{T^k}/\sim$.

\begin{defn} \label{dis-poly-notion}
    A symmetric polynomial \(g(x_1,\dots,x_n)\) in the elementary symmetric polynomials \(\sigma_1(x_1,\dots,x_n),\dots,\sigma_n(x_1,\dots,x_n)\) (i.e., a polynomial of equivariant Chern classes) is called a \textit{distinguishing polynomial} for \(M^{T^k}\) if for any two distinct equivalence classes \([p],[q]\in M^{T^k}/{\sim}\),
\begin{equation}\label{inequ}
g(\lambda_{p,1},\dots,\lambda_{p,n})\neq g(\lambda_{q,1},\dots,\lambda_{q,n})
\end{equation}
in $H^{*}(BT^k)$.

\end{defn}

The following lemma guarantees the existence of such a distinguishing polynomial.

\begin{lem}\label{polynomial}
 There always exists a distinguishing polynomial g that can be written as a linear combination of $\sigma_1,\dots,\sigma_n$.
\end{lem}

\begin{proof}
 For two different classes $[p]\not=[q]$  in $M^{T^k}/\sim$, we have that
 $ c^{T^k}(T_pM)=\prod_{i=1}^n(1+\lambda_{p,i})\not=\prod_{i=1}^n(1+\lambda_{q,i})=c^{T^k}(T_qM) $. So there must exist some integer $r$ with $1 \leq r \leq n$ such that
$$\sigma_{r}(\lambda_{p,1}, ..., \lambda_{p,n}) \neq \sigma_{r}(\lambda_{q,1}, ..., \lambda_{q,n}).$$
Let
$r(p,q)= \min \left\{ 1 \leq r \leq n\ \big|\ \sigma_r(\lambda_{p,1}, ..., \lambda_{p,n}) \neq \sigma_r(\lambda_{q,1}, ..., \lambda_{q,n})\right\}$ and set
$$\Delta (p, q) = \sigma_{r(p,q)}(\lambda_{p,1}, ..., \lambda_{p,n}) - \sigma_{r(p,q)}(\lambda_{q,1}, ..., \lambda_{q,n}).$$
Then $\Delta(p,q) \neq 0.$

On the other hand, since each $\lambda _{p,i}$ can be written as a linear combination  $\sum\limits_{j = 1}^k {{\lambda _{p,i,j}}{t_j}}$ in $H^*(BT^k)=\mathbb{Z}[t_1, ..., t_k]$,
$\Delta(p,q)$ becomes a nonzero polynomial function in
$\mathbb{Z}[t_1, ..., t_k]$, so is
$$F(t_1, ..., t_k)
=\prod_{[p]\not=[q]}\Delta(p,q).$$
Regarding $F(t_1,t_2,\dots,t_k)$ as a polynomial function on $\mathbb{R}^k$, we claim that there exists $(s_1,s_2,\dots,s_k) \in \mathbb{Z}^k$ such that $F(s_1,s_2,\dots,s_k) \neq 0$. Next,  we prove this claim by induction on the number  $k$, the rank of $T^k$.

When $k=1$, by the  Fundamental Theorem of Algebra, a unary polynomial $F(t_1)$ has at most finite roots in $\mathbb{R}$, so there must exist an integer $s_1 \in \mathbb{Z}$ such that $F(s_1) \neq 0$.

Assume that the claim holds for $k=l\geq 1$, and consider the case $k=l+1$. Write

$$F(t_1,t_2,...,t_{l+1}) = \sum\limits_{{\rm{i}} = 0}^d {G_i\left( {{t_1},{t_2},...,{t_l}} \right)t_{l + 1}^{h_i}} $$
where each  $G_i(t_1,t_2,\dots,t_l)$ is a nonzero   polynomials in $l$ variables, and $0\leq h_0<h_1<\cdots <h_d$. By  induction hypothesis, there exists $(s_1,s_2,\dots,s_l) \in \mathbb{Z}^l$ such that $G_d(s_1,s_2,\dots,s_l)$  $ \neq 0$. Thus $F(s_1,s_2,\dots,s_l,t_{l+1}) = f(t_{l+1})$ is a nonzero unary polynomial of degree $h_d$ in the variable $t_{l+1}$. Furthermore,  there exists $s_{l+1} \in \mathbb{Z}$ such that $f(s_{l+1}) = F(s_1,s_2,\dots,s_l,s_{l+1}) \neq 0$. This completes the  induction.

Now write  \(\widetilde{\lambda}_{p,i}=\sum_{j=1}^k\lambda_{p,i,j}s_j\) for each class  \([p]\in M^{T^k}/\sim\), and let
\[\mathbf{r}=\max_{[p]\not=[q]}\{r(p,q)\}.\]
We further set
\[\ell = \max_{1\leq i\leq \mathbf{r}} \left\{ \big|\sigma_i(\widetilde{\lambda}_{p,1}, \dots, \widetilde{\lambda}_{p,n})\big| \,\bigg|\, [p]\in M^{T^k}/\sim \right\} \in \mathbb{Z},\]
and let \(N = 2\ell+2 \in \mathbb{Z}\). Choose
\begin{equation}\label{distin-poly}
g(x_1, ..., x_n)= \sum_{i=1}^{\bf r} N^{i}\sigma_i(x_1, ..., x_n).
\end{equation}

Now let us show that $g$ is exactly the required symmetric polynomial. It suffices to prove that  for any two different classes $[p], [q]\in M^{T^k}/\sim$,
$$g'(\widetilde{\lambda}_{p,1}, ..., \widetilde{\lambda}_{p,n})\not=g'(\widetilde{\lambda}_{q,1}, ..., \widetilde{\lambda}_{q,n})$$
where $g'(x_1, ..., x_n)=g(x_1, ..., x_n)+(\ell+1)\sum_{i=1}^{\bf r}N^i$.
In fact, it is obvious that $$\sigma_i\left(\widetilde{\lambda}_{p,1}, ..., \widetilde{\lambda}_{p,n}\right) + \ell+1, \ \sigma_i\left(\widetilde{\lambda}_{q,1}, ..., \widetilde{\lambda}_{q,n}\right) + \ell+1 \in (0, N).$$
Since
$$\sigma_{r(p,q)}\left(\widetilde{\lambda}_{p,1}, ..., \widetilde{\lambda}_{p,n}\right)
\not= \sigma_{r(p,q)}\left(\widetilde{\lambda}_{q,1}, ..., \widetilde{\lambda}_{q,n}\right)$$
for all $r(p,q)\leq {\bf r}$ by the above claim, it follows immediately from the Uniqueness Theorem for $N$-adic representations~\cite[Theorem 136]{HDW}  that $g'(\widetilde{\lambda}_{p,1}, ..., \widetilde{\lambda}_{p,n})\not=g'(\widetilde{\lambda}_{q,1}, ..., \widetilde{\lambda}_{q,n})$, as desired.
\end{proof}


Now we are going to finish the proof of Theorem~\ref{main thm1}.

\begin{proof}[Proof of Theorem~\ref{main thm1}]
If $M$ bounds equivariantly, then it follows immediately from  Theorem~\ref{invariant1} that for any equivariant cohomology class $g$ in equivariant Chern classes,
$$\big\langle g^l, [M]\big\rangle=0$$ for $l\geq 0$.

Conversely, 
by Lemma~\ref{polynomial}, there exists a distinguishing polynomial $g(x_1, ..., x_n)$ 
satisfying  that
for any two different classes $[p], [q]\in M^{T^k}/\sim$,
$$g(\lambda_{p,1}, ..., \lambda_{p,n})\not=g(\lambda_{q,1}, ..., \lambda_{q,n})$$
in $H^*(BT^k)$.
Assume that
$\big\langle g(x_1, ..., x_n)^l, [M]\big\rangle=0$ for all $ l< |M^{T^k}|$.
Then it follows from the formula~\eqref{ABBV-1} and Remark~\ref{deg<2n} that
for all $ l< |M^{T^k}|$,
\begin{align}\label{eq}
 \big\langle g(x_1, ..., x_n)^l, [M]\big\rangle=
 \sum_{[p]\in M^{T^k}/\sim}{{\varepsilon_{[p]}g(\lambda_{p,1}, ...,\lambda_{p,n})^l}\over {
 \sigma_n(\lambda_{p, 1}, ...,\lambda_{p,n}))}}=0.
\end{align}
Write $M^{T^k}/\sim=\{[p_1], ..., [p_m]\}$, let $l$ run over $0, 1, ..., m-1<  |M^{T^k}|$, we then obtain a system of equations from (\ref{eq}) that
\begin{equation*}
\begin{pmatrix}
1 &  \cdots & 1\\
g(\lambda_{p_1,1}, ...,\lambda_{p_1,n}) &\cdots& g(\lambda_{p_m,1}, ...,\lambda_{p_m,n})\\
\vdots &  \ddots & \vdots\\
g(\lambda_{p_1,1}, ...,\lambda_{p_1,n})^{m-1} & \cdots & g(\lambda_{p_m,1}, ...,\lambda_{p_m,n})^{m-1}\\
\end{pmatrix}
\begin{pmatrix}
{{\varepsilon_{[p_1]}}\over {\sigma_n(\lambda_{p_1,1}, ...,\lambda_{p_1,n})}} \\
{{\varepsilon_{[p_2]}}\over {\sigma_n(\lambda_{p_2,1}, ...,\lambda_{p_2,n})}} \\
\vdots \\
{{\varepsilon_{[p_m]}}\over {\sigma_n(\lambda_{p_m,1}, ...,\lambda_{p_m,n})}}
\end{pmatrix}
=
\begin{pmatrix}
0 \\
0 \\
\vdots \\
0
\end{pmatrix}
\end{equation*}
whose coefficient matrix is a Vandermonde matrix, so $\varepsilon_{[p_1]}=\cdots=\varepsilon_{[p_m]}=0$.  Furthermore, it follows from Lemma~\ref{zero} and Theorem~\ref{invariant1} that $M$ bounds equivariantly.
\end{proof}

The following example  illustrates that the top equivariant Chern class is not a valid choice for $g$ in Theorem~\ref{main thm1} when $n>1$.

\subsection{Counterexample: top equivariant Chern class is insufficient}
    \label{Example}
Given an integer $a$, let $\phi_a: S^1\times \mathbb{C}P^1\longrightarrow \mathbb{C}P^1$ be the $S^1$-action defined by $(s, [z_0: z_1])\longmapsto [z_0:s^az_1]$, which fixes two points $[1:0]$ and $[0:1]$ with weights $\{a\}$ and $\{-a\}$, respectively. Then,     the diagonal action of $\phi_a$ and $\phi_b$ gives a $S^1$-action $$\phi_{ab}: S^1\times \mathbb{C}P^1\times  \mathbb{C}P^1\longrightarrow \mathbb{C}P^1\times  \mathbb{C}P^1$$ by
$(s, [z_0:z_1], [z'_0, z'_1])\longmapsto ( [z_0:s^az_1], [z'_0, s^bz'_1]),$
which has four fixed points $$p_1=([1:0], [1:0]), p_2=([1:0], [0:1]), p_3=([0:1], [1:0]), p_4=([0:1], [0:1]),$$ and the corresponding weights at the fixed points are $\{a,b\}, \{a, -b\}, \{-a, b\}, \{-a,-b\}$, respectively. It is not difficult to see that the action $\phi_{ab}$ preserves the almost complex structure on $\mathbb{C}P^1\times  \mathbb{C}P^1$. By
$M(a,b)$ we denote the   $\mathbb{C}P^1\times  \mathbb{C}P^1$ with the $S^1$-action
$\phi_{ab}$.

Now we can easily read off the total equivariant Chern classes of the tangent representations at the four fixed points, as shown in the table below.
\[
\begin{array}{c|l}
\text{Fixed point}\ p & c^{S^1}(T_p M(a,b))\in H^*(BS^1)=\mathbb{Z}[t] \\
\hline
p_1 &\ \ \ \ \ \ \ (1+at)(1+bt) \\
p_2 &\ \ \ \ \ \ \ (1+at)(1-bt) \\
p_3 &\ \ \ \ \ \ \ (1-at)(1+bt) \\
p_4 &\ \ \ \ \ \ \ (1-at)(1-bt).
\end{array}
\]
Using ABBV localization formula, a direct calculation shows that for any integer $l\geq 0$
$$\langle (c_2^{S^1}(TM(a,b)))^l, [M(a,b)] \rangle =2(abt^2)^{l-1}+2(-abt^2)^{l-1}
$$
and
$$\langle (c_1^{S^1})^2c_2^{S^1}(TM(a,b)), [M(a,b)] \rangle =4(a^2+b^2)t^2.
$$
Take $(a, b)=(1, 6)$ and $(2,3)$, respectively, we have that
$$\langle (c_2^{S^1}(TM(1,6)))^l, [M(1,6)] \rangle=\langle (c_2^{S^1}(TM(2,3)))^l, [M(2,3)] \rangle=2(6t^2)^{l-1}+2(-6t^2)^{l-1}$$
always holds,
but
\begin{align*}
  \langle (c_1^{S^1})^2c_2^{S^1}(TM(1,6)), [M(1,6)] \rangle =148t^2
 \not=  52t^2=
\langle (c_1^{S^1})^2c_2^{S^1}(TM(2,3)), [M(2,3)] \rangle,\end{align*}
implying that $M(1,6)$ is not equivariantly bordant to $M(2,3)$. This reveals that generally,  the top equivariant Chern class
cannot be chosen as  a distinguishing polynomial  when $n>1$.

\vskip .2cm
We end this subsection by calculating \(\mathbf{m}(1,M(a,b))\), demonstrating its computability.
Recall that the minimal distinguishing degree $${\bf m}(k, M^{2n})=\min\{\deg g\big| g \text{ is a distinguishing  polynomial}\}. $$
From the above table, we  write out
 their first equivariant Chern classes of the tangent representations at the four fixed points, as listed in the table below.
\[
\begin{array}{c|l}
\text{Fixed point } p & c_1^{S^1}(T_p M(a,b)) \\
\hline
p_1 &\ \ \ (a+b)t \\
p_2 &\ \ \ (a-b)t \\
p_3 &\ \ \ (b-a)t \\
p_4 &\ -(a+b)t.
\end{array}
\]
When $a\not= b$, these first equivariant Chern classes are all distinct. By Definition~\ref{dis-poly-notion},  $g=c_1^{S^1}(TM(a,b))$ is a distinguishing polynomial. When $a=b$, the fixed point set $M(a,a)^{S^1}$ is divided into three equivalence classes $[p_1], [p_2]=[p_3], [p_4]$ since
$$c^{S^1}(T_{p_2} M(a,b))=c^{S^1}(T_{p_3} M(a,b))=1-a^2t^2.$$ In this case, obviously, $g=c_1^{S^1}(TM(a,b))$ is also a distinguishing polynomial. Thus, we conclude that \(\mathbf{m}(1,M(a,b))=2\) regardless of whether \(a\neq b\) holds or not.

\subsection{The case  $k=n$}\label{Case-k=n}
In what follows, we present explicit admissible forms of $g$ stated in Theorem~\ref{main thm1} for the case $k=n$, demonstrating that multiple valid choices of $g$ are available.

Let~$M^{2n}$ be a $2n$-dimensional unitary  $T^n$-manifold with a finite fixed point set $M^{T^n}$.
Take a fixed  point $p\in M^{T^n}$. Since
the action of $T^n$ on $M^{2n}$ is assumed to be effective,
the tangent representation of  $T_pM$ is faithful, so $n$ factors of degree 2 in $c_n^{T^n}(T_pM)=\prod_{i=1}^n\lambda_{p,i}\in H^\ast(BT^n)$ are linearly independent in $H^2(BT^n)$, and then $c_1^{T^n}(T_pM)=\sigma_1(\lambda_{p,1}, ..., \lambda_{p,n})\not=0$.

\begin{lem}\label{n=k}
 Let $p,q\in M^{T^n}$. Then  the following statements are equivalent.
 \begin{itemize}
     \item[(1)]
 $c^{T^n}(T_pM)=c^{T^n}(T_qM)$;
 \item[(2)]
$c_1^{T^n}(T_pM)=c_1^{T^n}(T_qM)$ and $c_2^{T^n}(T_pM)=c_2^{T^n}(T_qM)$;
\item [(3)]
$c_1^{T^n}(T_pM)=c_1^{T^n}(T_qM)$ and $c_n^{T^n}(T_pM)=c_n^{T^n}(T_qM)$.
  \end{itemize}
\end{lem}

\begin{proof}
The equivalence  $(1) \Longleftrightarrow (2)$  is precisely established by
\cite[Lemma 3.2]{LT1}.
The proof of  $(1) \Longrightarrow (3)$ is obvious. Thus, it suffices to show that
$(3) \Longrightarrow (2)$. This is equivalent to showing $c_2^{T^n}(T_pM)=c_2^{T^n}(T_qM)$ if (3) holds. For $n=1$, $c_2^{T^n}(T_pM)$ and $c_2^{T^n}(T_qM)$ naturally degenerates to zero, so we may assume $n>1$.
From
 $$c_n^{T^n}(T_pM)=\prod_{i=1}^n\lambda_{p,i}=\prod_{i=1}^n\lambda_{q,i} =c_n^{T^n}(T_qM),$$ without loss of generality,   assume that  $\lambda_{q, i}=a_i\lambda_{p,i}, i=1, ..., n$, where $a_i=\pm 1$.  Then it follows
 $$s_2(\lambda_{p,1}, ..., \lambda_{p,n})=\sum_{i=1}^n\lambda_{p,i}^2=
\sum_{i=1}^n\lambda_{q,i}^2= s_2(\lambda_{q,1}, ..., \lambda_{q,n}).$$
By Newton formula (cf~\cite{MS1}), we know that $s_2=\sigma_1^2-2\sigma_2$, so
$$(c_1^{T^n}(T_pM))^2-2c_2^{T^n}(T_pM)=(c_1^{T^n}(T_qM))^2-2c_2^{T^n}(T_qM).$$
 This gives $c_2^{T^n}(T_pM)=c_2^{T^n}(T_qM)$ since $c_1^{T^n}(T_pM)=c_1^{T^n}(T_qM)$.
\end{proof}

By the equivalent conditions in Lemma~\ref{n=k}, we obtain two  explicit  admissible expressions of $g$ when $k=n$, which  yields the following equivariant bordism criterion.

\begin{prop}\label{gn=k}
 Let $M^{2n}$ be a $2n$-dimensional unitary $T^n$-manifold with a finite fixed point set $M^{T^k}$.
Then $M^{2n}$ is an equivariant boundary precisely when
$$\langle g^l, [M]\rangle = 0$$
for all integers $l < |M^{T^n}|$, with two admissible choices
$g = c_1^{T^n}(TM)+c_2^{T^n}(TM) $ or $ g = c_1^{T^n}(TM)+c_n^{T^n}(TM)$.
\end{prop}

\begin{proof}
 By the proof of Theorem~\ref{main thm1}, it suffices to show that $c_1^{T^n}(TM)+c_2^{T^n}(TM) $ and $  c_1^{T^n}(TM)+c_n^{T^n}(TM)$ are two distinguishing polynomials. This is equivalent to proving that
 for two different  classes $[p], [q]\in M^{T^n}/\sim$,
  $$c_1^{T^n}(T_pM)+c_2^{T^n}(T_pM)\not=c_1^{T^n}(T_qM)+c_2^{T^n}(T_qM)$$
  and
 $$c_1^{T^n}(T_pM)+c_n^{T^n}(T_pM)\not=c_1^{T^n}(T_qM)+c_n^{T^n}(T_qM).$$
 If not, assume that  $c_1^{T^n}(T_pM)+c_2^{T^n}(T_pM)=c_1^{T^n}(T_qM)+c_2^{T^n}(T_qM)$. Then it  immediately follows that   $c_1^{T^n}(T_pM)$ $ = c_1^{T^n}(T_qM)$ and $c_2^{T^n}(T_pM)=c_2^{T^n}(T_qM)$ from the degree properties of the graded ring $H^*(BT^n)$. Similarly,
 if $$c_1^{T^n}(T_pM)+c_n^{T^n}(T_pM)=c_1^{T^n}(T_qM)+c_n^{T^n}(T_qM),$$
 then we conclude that $c_1^{T^n}(T_pM)=c_1^{T^n}(T_qM)$ and $c_n^{T^n}(T_pM)=c_n^{T^n}(T_qM)$.
 Furthermore, in either case, it follows from Lemma~\ref{n=k} that
 $c^{T^n}(T_pM)=c^{T^n}(T_qM)$, implying that $[p]=[q]$, a contradiction.
\end{proof}

Using  $g = c_1^{T^n}(TM)+c_2^{T^n}(TM)$ of degree 4 in Proposition~\ref{gn=k}, one derives that $\boldsymbol{m}(n, M^{2n})\leq 4$. With Corollary~\ref{md-deg inequ} together, we conclude that
\begin{cor} \label{conj-solution}  Suppose that~$M^{2n}$ is  a  $2n$-dimensional  unitary  $T^n$-manifold with a finite fixed point set $M^{T^n}$.  Then $\boldsymbol{m}(n, M^{2n})\leq 4$. Moreover, $2\chi(M)>n$ if $M^{2n}$ does not bound equivariantly.
  \end{cor}

  The following example shows that $\boldsymbol{m}(n, M^{2n})=2$ is attainable.

  \begin{example}
   Consider the standard $T^n$-action on the complex projective space  $\mathbb{C}P^n$  defined by
   $$((g_1, ..., g_n), [z_0:z_1: \cdots : z_n])\longmapsto [z_0:g_1z_1: \cdots : g_nz_n].$$
   This action has \(n+1\) fixed points \(p_0,p_1,\dots,p_n\),
   where
$$p_i = [\underbrace{0:\dots:0}_{i}:1:0:\dots:0],\quad i=0,1,\dots,n. $$
By a direct calculation, we obtain that the total equivariant Chern class of the tangent representation at $p_i$ is
\[
c^{T^n}(T_{p_i}\mathbb{C}P^n) =
\begin{cases}
\prod_{j=1}^n (1+t_j)
& \text{ if } i=0\\
(1-t_i)\prod_{\substack{1\le j\le n\\j\ne i}}\big(1+t_j-t_i\big) & \text{ if } i\not=0
\end{cases}
\]
from which we have
\[
c_1^{T^n}(T_{p_i}\mathbb{C}P^n) =
\begin{cases}
\sum_{j=1}^n (1+t_j)
& \text{ if } i=0\\
 (\sum_{j=1}^n t_j)-(n+1)t_i
 & \text{ if } i\not=0.
\end{cases}
\]
Clearly, all $n+1$  first equivariant Chern classes are distinct. In fact, linear independence of  $t_1,\dots,t_n$ as generators in $H^*(BT^n)=\mathbb{Z}[t_1, ..., t_n]$ implies that \(c_1^{T^n}(T_{p_i}\mathbb{C}P^n)\neq c_1^{T^n}(T_{p_j}\mathbb{C}P^n)\) whenever \(i\neq j\).
Thus, $g=c_1^{T^n}(T\mathbb{C}P^n)$ is a distinguishing polynomial and is of degree 2. Thus, we conclude that ${\bf m}(n, \mathbb{C}P^n)=2$.
\end{example}

\subsection{The case $k=n-1$}\label{Case-k=n-1}
Let~$M^{2n}$ be a $2n$-dimensional unitary  $T^{n-1}$-manifold with a finite fixed point set $M^{T^{n-1}}$. Wen and Ma proved the following key lemma in \cite{WM}.

\begin{lem}\label{k=n-1}
 Let $p,q\in M^{T^{n-1}}$. Then
 $c^{T^{n-1}}(T_pM)=c^{T^{n-1}}(T_qM)$
 if and only if
$c_i^{T^{n-1}}(T_pM)=c_i^{T^{n-1}}(T_qM)$ for all $i\leq 6$.
 \end{lem}

For any fixed point $p\in M^{T^{n-1}}$,  each equivariant Chern class $ c_i^{T^{n-1}}(T_pM)$ is of degree $2i$.  This degree constraint forces Lemma \ref{k=n-1} to hold only for \(n\ge 6\). Consequently, the polynomial
\(g=\sum_{i=1}^6 c_i^{T^{n-1}}(TM)\)
is an explicit distinguishing polynomial of  degree 12. Therefore,  we obtain by Corollary~\ref{md-deg inequ} that

\begin{cor}\label{weak-ineq}
Let~$M^{2n}$ be a $2n$-dimensional unitary  $T^{n-1}$-manifold with a finite fixed point set. Then
  ${\bf m}(n-1, M^{2n})\leq 12$.  Moreover, $6\chi(M^{2n})>n$ provided that $M^{2n}$ does not bound equivariantly.
\end{cor}

On the other hand, when $n=3$,  as noted by Wen and Ma in~\cite[Remark 6]{WM},  the 6-sphere \(S^6 \cong G_2/SU(3)\) admits a \(G_2\)-invariant almost complex structure. The canonical action of a common maximal torus \(T^2\)  on \(S^6\), which preserves this almost complex structure,   has exactly two isolated fixed points.
This explicit homogeneous $T^2$-manifold of dimension 6 serves as a concrete example supporting the inequality $2\chi(M)>n$.

Now let us look at  ${\bf m}(2, S^6)$. Since $\chi(S^6)=2$, we then see from the inequality (\ref{important-inequ})
that ${\bf m}(2, S^6)\chi(S^6)=2{\bf m}(2, S^6)>6$ so ${\bf m}(2, S^6)>3$. By direct calculation, we can  show that
${\bf m}(2, S^6)=6$. In fact, it is well-known that the weight sets of the tangent representations at two fixed points $p_+$ and $p_-$ are
$$W_+=\{(1,0),\ (-1,1),\ (0,-1)\}, \ \
W_-=\{(-1,0),\ (1,-1),\ (0,1)\}$$
respectively. Then the total equivariant Chern classes of the tangent representations $T_{p_+}S^6$ and $T_{p_-}S^6$ are
$$c^{T^2}(T_{p_+}S^6)
=(1+t_1)(1-t_1+t_2)(1-t_2)$$
and
$$c^{T^2}(T_{p_-}S^6)
=(1-t_1)(1+t_1-t_2)(1+t_2)$$
respectively, in $H^*(BT^2)=\mathbb{Z}[t_1,t_2]$. Furthermore, it follows that
$$c_1^{T^2}(T_{p_+}S^6)=
c_1^{T^2}(T_{p_-}S^6)=0, \ \ \
c_2^{T^2}(T_{p_+}S^6)=
c_2^{T^2}(T_{p_-}S^6)=t_1t_2-t_1^2-t_2^2,
$$
and
$c_3^{T^2}(T_{p_+}S^6)=
t_1t_2(t_1-t_2)\not=-t_1t_2(t_1-t_2)=
c_3^{T^2}(T_{p_-}S^6)$, thus concluding that ${\bf m}(2, S^6)=6$. Again using the inequality
(\ref{important-inequ}), we derive that ${\bf m}(2, S^6)\chi(S^6)=6\chi(S^6)>6$, so $\chi(S^6)>1$, giving a sharp inequality  and supporting $2\chi(S^6)>3$ since $\chi(S^6)=2$.

This example illustrates that the inequality
${\bf m}(k, M^{2n})\chi(M)>2n$
in
(\ref{important-inequ})
may lead to a sharp inequality even for ${\bf m}(k, M^{2n})>4$.

\section{Equivariant bordism for unoriented  $(\mathbb{Z}_2)^k$-manifolds}\label{proof2}

In this section, we give the proof of Theorem~\ref{main thm2}.

\begin{lem}\label{equ}
Let $M^n$ be an $n$-dimensional smooth closed $(\mathbb{Z}_2)^k$-manifold with  a finite fixed point set $M^{(\mathbb{Z}_2)^k}$. Then $e^{(\mathbb{Z}_2)^k}(T_pM), p\in M^{(\mathbb{Z}_2)^k}$ appear in pairs if and only if
$$\sum_{p\in M^{(\mathbb{Z}_2)^k}}(e^{(\mathbb{Z}_2)^k}(T_pM))^h=0$$
for all $0\leq h<|M^{(\mathbb{Z}_2)^k}|$.
\end{lem}

\begin{proof}
If all $e^{(\mathbb{Z}_2)^k}(T_pM), p\in M^{(\mathbb{Z}_2)^k}$ appear in pairs, then it is obvious that $$\sum_{p\in M^{(\mathbb{Z}_2)^k}}(e^{(\mathbb{Z}_2)^k}(T_pM))^h=0$$
for all $h\geq 0$.

Now suppose that $\sum_{p\in M^{(\mathbb{Z}_2)^k}}(e^{(\mathbb{Z}_2)^k}(T_pM))^h=0$
for all $0\leq h<|M^{(\mathbb{Z}_2)^k}|$. If all $e^{(\mathbb{Z}_2)^k}(T_pM)$, $p\in M^{(\mathbb{Z}_2)^k}$ do not appear in pairs, then there must exist a subset $\mathcal{S}=\{p_{i_1}, ..., p_{i_m}\}$ of $ M^{(\mathbb{Z}_2)^k}$ with
$m\leq |M^{(\mathbb{Z}_2)^k}|$
such that all $e^{(\mathbb{Z}_2)^k}(T_pM), p\in \mathcal{S}$ are distinct, and $\sum_{p\in M^{(\mathbb{Z}_2)^k}}(e^{(\mathbb{Z}_2)^k}(T_pM))^h=\sum_{p\in \mathcal{S}}(e^{(\mathbb{Z}_2)^k}(T_pM))^h=0$
for all $0\leq h<|M^{(\mathbb{Z}_2)^k}|$.

Let $h$ run over $0, 1, ..., m-1$,
  we then obtain
\begin{equation*}
\begin{pmatrix}
1&  \cdots & 1\\
e^{(\mathbb{Z}_2)^k}(T_{p_{i_1}}M) &\cdots& e^{(\mathbb{Z}_2)^k}(T_{p_{i_m}}M)\\
\vdots &  \ddots & \vdots\\
(e^{(\mathbb{Z}_2)^k}(T_{p_{i_1}}M))^{m-1} & \cdots & (e^{(\mathbb{Z}_2)^k}(T_{p_{i_m}}M))^{m-1}\\
\end{pmatrix}
\begin{pmatrix}
1 \\
1 \\
\vdots \\
1
\end{pmatrix}
=
\begin{pmatrix}
0 \\
0 \\
\vdots \\
0
\end{pmatrix}
\end{equation*}
whose coefficient matrix is a Vandermonde matrix, giving a contradiction. Thus, $e^{(\mathbb{Z}_2)^k}(T_pM), p\in M^{(\mathbb{Z}_2)^k}$ appear in pairs.
\end{proof}

Next let us prove  Theorem~\ref{main thm2}.

\begin{proof}[Proof of Theorem~\ref{main thm2}]
Suppose that $M^n$ is an $n$-dimensional smooth closed $(\mathbb{Z}_2)^k$-manifold with  isolated fixed points.

Obviously, if $M^n$ bounds equivariantly, then by Theorem~\ref{invariant2}, $\langle (w_n^{(\mathbb{Z}_2)^k}(TM))^l, [M]\rangle=0$ for all $l\in \mathbb{N}$.

Conversely, assume that $\langle (w_n^{(\mathbb{Z}_2)^k}(TM))^l, [M]\rangle=0$ for all $ l \leq |M^{(\mathbb{Z}_2)^k}|$. Using the localization formula in Theorem~\ref{DKS}, we have that
\begin{align*}
 \langle (w_n^{(\mathbb{Z}_2)^k}(TM))^l, [M]\rangle = &\langle (\sigma_n(x_1, ..., x_n))^l, [M]\rangle  \\
 =& \sum_{p\in M^{(\mathbb{Z}_2)^k}}
 (\sigma_n(\lambda_{p,1}, ..., \lambda_{p,n}))^{l-1}\\
 = & \sum_{p\in M^{(\mathbb{Z}_2)^k}}
 (e^{(\mathbb{Z}_2)^k}(T_pM))^{l-1}\\
 =& \ 0.
\end{align*}
By Lemma~\ref{equ}, all $e^{(\mathbb{Z}_2)^k}(T_pM), p\in M^{(\mathbb{Z}_2)^k}$ appear in pairs. This means by the localization formula in Theorem~\ref{DKS} that all equivariant Stiefel--Whitney numbers vanish. Moreover, it
 follows from Theorem~\ref{invariant2} that $M^n$ bounds equivariantly.
\end{proof}

\begin{cor}\label{pair}
Let $M^n$ be an $n$-dimensional smooth closed $(\mathbb{Z}_2)^k$-manifold with a finite fixed point set. Then $M^n$ bounds equivariantly  if and only if the equivariant Euler classes $e^{(\mathbb{Z}_2)^k}(T_pM)$ for all $p\in M^{(\mathbb{Z}_2)^k}$ appear in pairs.
\end{cor}

\begin{rem}\label{Stong}
Corollary \ref{pair} provides a new criterion for the vanishing of the equivariant bordism class of \(M^n\), formulated via  the equivariant Euler classes of tangent  representations at all fixed points.  In contrast, a classical theorem of Stong~ \cite{St} asserts that \(M^n\) bounds equivariantly  precisely when the tangent representations \(T_pM\) over \(p\in M^{(\mathbb {Z}_2)^k}\) occur in  pairs, up to isomorphism. When \(k=n\), the vanishing of the equivariant bordism class can also alternatively be characterized using the dual representations of these tangent modules; relevant details can be found in \cite{CLT,LT2}.
 \end{rem}

\section{Applications to Kosniowski’s Conjecture and its toric generalization}
\label{apply-KC}

This section uses
Corollaries~\ref{md-deg inequ} and \ref{sharp bound} to investigate Kosniowski’s conjecture and its toric generalization.

\begin{conj}[Kosniowski~\cite{K}]\label{C-K}
 Suppose that $M^{2n}$ is a unitary $S^1$-manifold with isolated fixed points. If $M^{2n}$ is not a boundary, then  the number of fixed points is  greater than  a positive linear function $f(n)$ of $n$.
 \end{conj}

 \begin{rem}
  Kosniowski conjectured that the candidate function is \(f(n) = \tfrac{n}{2}\). Moreover, Wen and Ma extended 
  Conjecture~\ref{C-K} to the \(T^k\)-equivariant setting, formulating a generalized version recorded as Conjecture 2 in \cite{WM}, which we here refer to as the {\em toric generalization}.
\end{rem}

Rearranging the inequality \eqref{important-inequ} gives
$$
\chi(M^{2n}) > \dfrac{2n}{\boldsymbol{m}(k, M^{2n})},
$$
which establishes a direct connection to Kosniowski's  conjecture and its toric generalization.
Our rearranged inequality  suggests taking $f(n)=\dfrac{2n}{{\bf m}(k, M^{2n})}$, thus settling the existence 
 of  a positive linear function $f(n)$
 that provides  a lower bound for  $\chi(M)$ in  both Kosniowski's conjecture and its toric generalization.  We formalize this result as the following proposition.

\begin{prop}\label{ex-lin-bound}
    Suppose that $M^{2n}$ is a unitary $T^k$-manifold with isolated fixed points. If $M^{2n}$ is not a boundary, then    \begin{equation}\label{important-inequ-c}
\chi(M^{2n}) > \dfrac{2n}{\boldsymbol{m}(k, M^{2n})}.
\end{equation}
\end{prop}

 \begin{rem}
Substituting the candidate function $f(n)=\tfrac{n}{2}$, both Kosniowski’s conjecture and its toric generalization  reduce to the inequality $$2\chi(M^{2n})>n.$$
Meanwhile,  our
 inequality $\boldsymbol{m}(k, M^{2n})\chi(M^{2n})>2n$ in \eqref{important-inequ} also establishes  $2\chi(M^{2n})>n$ whenever   $\boldsymbol{m}(k, M^{2n})\leq 4$.
 Indeed, we know from \cite[Theorem IV.10.5]{B} and \cite[Lemma 7.4.3]{BP1}   that any closed connected smooth orientable (or unitary) manifold $N$ admitting an effective \(T^k\)-action with \(k\ge1\) contains a circle subgroup \(S\subset T^k\) satisfying \(N^S=N^{T^k}\). Consequently, the validity of \(2\chi(M^{2n}) > n\) for all unitary \(S^1\)-manifolds automatically implies the same bound for arbitrary unitary \(T^k\)-manifolds with \(k>1\).

\end{rem}

When $k=1$, partial progress supporting this inequality $2\chi(M^{2n}) > n$ has  been achieved  in~\cite{CKP, J, Ko1,  LL, LM, PT, W} under restrictions such as $\chi(M^{2n})\leq 3$ or additional geometric and algebraic  constraints for unitary $S^1$-manifolds. We note that the key lemma \cite[Lemma 3.1]{L-gap} employed to attack Kosniowski’s conjecture contains a logical gap, which was identified by the second author of this paper.

 When $k=n$,   Corollary \ref{conj-solution} establishes the inequality \(2\chi(M^{2n}) > n\) using the bounds of $\boldsymbol{m}(k, M^{2n})$, which gives an affirmative answer to the toric generalization of Kosniowski’s conjecture.
 This same inequality is also derived independently in \cite{LT1}.
We summarize as follows.

\begin{prop}
   The toric generalization of Kosniowski’s conjecture holds for $k=n$.
\end{prop}

By contrast, Wen and Ma treated the case \(k=n-1\) and proved in \cite[Theorem 3]{WM} the weaker lower bound \(\lceil \tfrac{n}{6} \rceil + 1\) for \(\chi(M^{2n})\), which only gives the looser estimate \(6\chi(M^{2n}) > n\) (also see Corollary~\ref{weak-ineq}),  where \(\lceil\cdot\rceil\) denotes the ceiling function.  In fact,  as seen in Subsection~\ref{Case-k=n-1}, their key technical lemma (see also Lemma~\ref{Case-k=n-1})
merely guarantees the existence of a distinguishing polynomial of degree 12. In Subsection \ref{Case-k=n-1}, we examine the unitary \(T^2\)-manifold \(S^6\cong G_2/SU(3)\) and verify that the value \(\boldsymbol{m}(2,S^6)=6\) is achievable. Substituting this into our main inequality \(\boldsymbol{m}(2,S^6)\chi(S^6) > 6\) simplifies to the sharp bound \(\chi(S^6) > 1\),  consistent with the estimate $2\chi(M^{2n})>n$.
This example shows that the inequality
${\bf m}(k, M^{2n})\chi(M)>2n$
in
(\ref{important-inequ})
may lead to a sharp inequality even for ${\bf m}(k, M^{2n})>4$, from which we can further induce $2\chi(M^{2n})>n$.
Moreover, this example also  reveals a limitation of Kosniowski’s candidate function $f(n)={n\over 2}$.

In what follows, we can show from  Corollary~\ref{sharp bound} that  \(f(n)=\dfrac{n}{2}\) may not be a sharp lower bound. To proceed, let us restate Corollary~\ref{sharp bound} below and furnish its proof.

\begin{cor}\label{sharp bound-1}
 Suppose   $M^{2n}$ is a  $2n$-dimensional unitary $T^k$-manifold with a finite fixed point set, satisfying further that  $M^{2n}$ bounds non-equivariantly, but does not bound equivariantly. Then  \begin{equation}\label{sharp ineq-1}
   {\bf m}(k, M^{2n})\chi(M^{2n})>2n {}+{\bf m}(k, M^{2n}).\end{equation}
\end{cor}
\begin{proof}
First, it follows from the ABBV formula that for a partition $\omega=(i_1, ..., i_r)$ with $i_1+\cdots+i_r=n$,
$$\langle c_\omega^{T^k}(TM), [M]\rangle= \langle  c_\omega (TM), [M]\rangle$$
where $c_\omega(TM)$ denotes a cup product of ordinary cohomology Chern classes in $H^*(M)$. Thus, if $M$ bounds non-equivariantly, then the pairing $\langle c_\omega^{T^k}(TM), [M]\rangle= \langle c_\omega (TM), [M]\rangle$ vanishes. Furthermore, suppose there exists a distinguishing polynomial $g$ with degree satisfying
$\deg g \cdot(\chi(M)-1)\leq 2n$. Then for every integer $0\leq l\leq \chi(M)-1$,
$$\langle g^l, [M]\rangle=0.$$
By Theorem~\ref{main thm1}, this  implies that $M$ bounds equivariantly, which yields a contradiction.
Hence, the required inequality (\ref{sharp ineq-1}) follows from the assumptions of the corollary.
\end{proof}

Applying Corollary~\ref{sharp bound-1}, we can obtain a  lower bound sharper than ${n\over 2}$
whenever ${\bf m}(k,M^{2n})\leq 4$. This partially settles Kosniowski’s conjecture and its toric generalization, provided that a unitary \(T^k\)-manifold \(M^{2n}\) with isolated fixed points  bounds non-equivariantly but not equivariantly and satisfies \(\boldsymbol{m}(k,M^{2n})\leq 4\).

\begin{prop}\label{sharper bound}
    Under the hypotheses of Corollary \ref{sharp bound-1}, whenever ${\bf m}(k,M^{2n})\leq 4$,  we have that
    $$\chi(M^{2n})> {n\over 2}+1.$$
    Moreover, Kosniowski's conjecture and its toric generalization hold.
\end{prop}

\begin{proof}
    Using our minimal distinguishing degree,  we derive the following  refined strict inequality by rearranging the inequality \eqref{sharp ineq-1}
$$\chi(M^{2n}) > \frac{2n}{m(k,M^{2n})} + 1,$$
which is significantly stronger than the  candidate function $f(n)={n\over 2}$ conjectured by Kosniowski, as shown below:
$$\chi(M^{2n})>{{2n}\over {{\bf m}(k, M^{2n})}}+1\geq {n\over 2}+1>{n\over 2}=f(n)$$
 whenever ${\bf m}(k,M^{2n})\leq 4$.
\end{proof}

\begin{rem}
 We attempted to establish the inequality \(\chi(M^{2n})> \frac{n}{2}+1\) without the hypothesis \(m(k,M^{2n})\le 4\), but our efforts did not yield a complete proof.
\end{rem}



\end{document}